\documentclass{daj}

\usepackage{amsmath,amsfonts,amsthm}
\usepackage[a4paper]{geometry}
\usepackage[utf8]{inputenc}
\usepackage[english]{babel}
\usepackage{microtype}
\usepackage{lmodern}
\usepackage{tikz}
\usepackage{bm}

\usepackage{enumerate}
\usepackage{xspace}

\usepackage[abbrev,msc-links]{amsrefs}

\newcommand{\Prim}{\mathcal{P}}
\newcommand{\ZZ}{\mathbb{Z}}
\newcommand{\NN}{\mathbb{N}}
\newcommand{\RR}{\mathbb{R}}
\newcommand{\CC}{\mathbb{C}}
\newcommand{\EE}{\mathbb{E}}
\newcommand{\PP}{\mathbb{P}}
\newcommand{\bv}{\mathbf v}

\newcommand{\bn}{\mathbf n}
\bmdefine{\bbeta}{\beta}

\newcommand{\Jarnik}{Jarn\'{\i}k\xspace}

\newcommand{\Z}{\ZZ}
\newcommand{\N}{\NN}

\newtheorem{lemma}{Lemma}[section]
\newtheorem{theorem}{Theorem}
\newtheorem{corollary}[lemma]{Corollary}

\dajAUTHORdetails{%
  title = {On the number of lattice convex chains}, 
  author = {Julien Bureaux and Nathana\"el Enriquez},
  plaintextauthor = {Julien Bureaux and Nathanael Enriquez},
}   

\dajEDITORdetails{%
   year={2016},
   number={19},
   received={4 June 2016},   
   revised={31 October 2016},    
   published={12 December 2016},  
   doi={10.19086/da1013},       
}   

\begin{document}

\begin{frontmatter}[classification=text]

\title{On the number of lattice convex chains}

\author[J. Bureaux]{Julien Bureaux}
\author[N. Enriquez]{Nathana\"el Enriquez}

\begin{abstract}
An asymptotic formula is presented for the number of planar lattice convex
polygonal lines joining the origin to a distant point of the diagonal.
The formula involves the non-trivial zeros of the zeta function and
leads to a necessary and sufficient condition for the Riemann Hypothesis to hold.
\end{abstract}
\end{frontmatter}

A convex chain with $k$ segments is defined as a sequence of points $(x_i,y_i)_{0\leq
i\leq k}$ of $\Z^2$ with $0=x_0<x_1<x_2<\cdots<x_{k-1}\leq x_k$, $0=y_0\leq y_1<y_2<\cdots<y_{k-1}<y_k$ and
\[
0\leq \frac{y_1-y_0}{x_1-x_0}<\cdots<\frac{y_k-y_{k-1}}{x_k-x_{k-1}}\leq +\infty
\]
(see Figure~\ref{fig:chain}).
The point $(x_k,y_k)$ will be called the endpoint of the chain. 
This paper deals with the enumeration of convex chains with endpoint $(n,n)$ on the diagonal, where the number of segments is not fixed. 

\begin{figure}[h]
\centering
\includegraphics[scale=0.65]{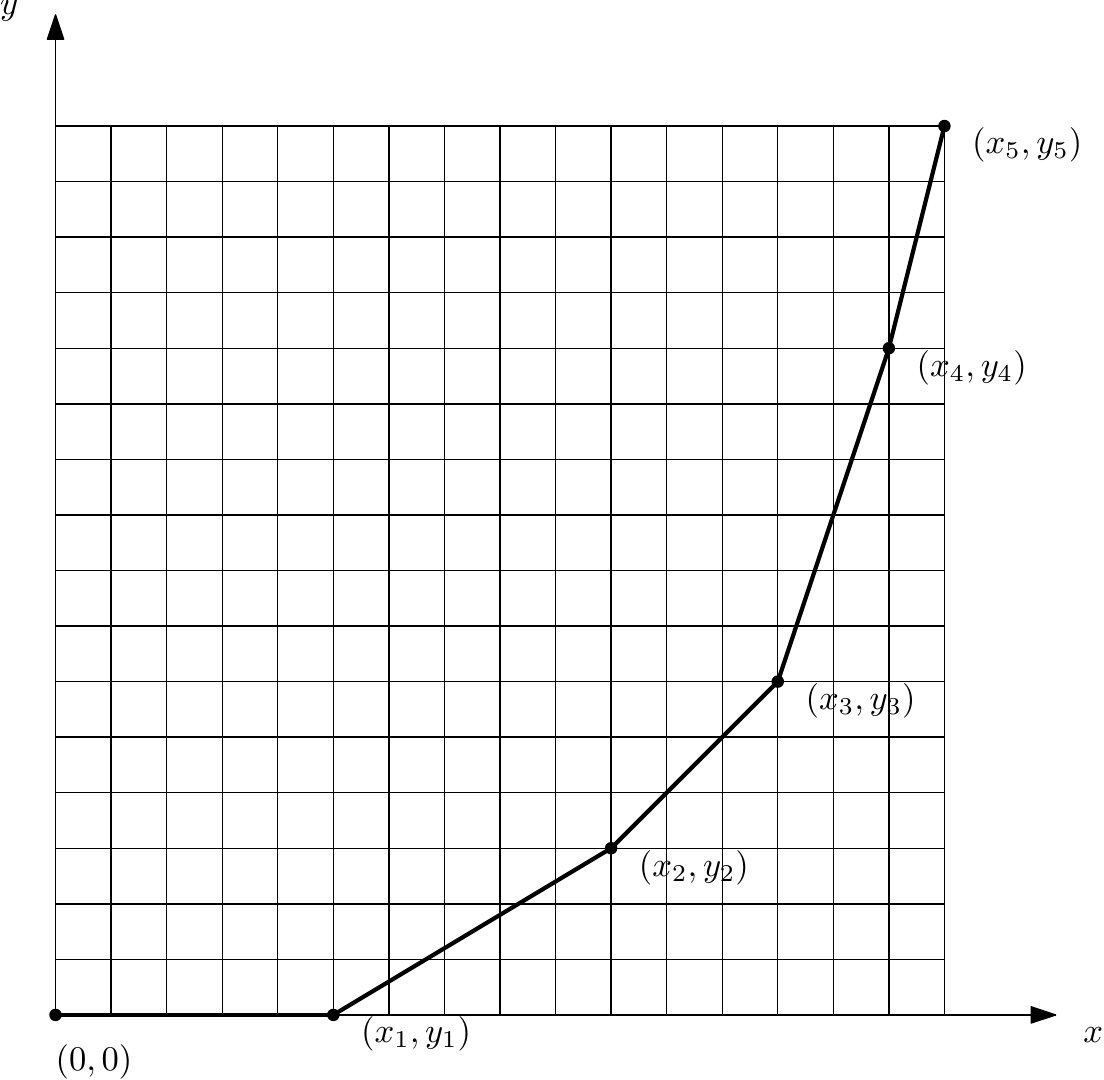}
\caption{Example of a convex chain with 5 segments}
\label{fig:chain}
\end{figure}

In 1979, Arnold~\cite{arnold_statistics_1980} considered the question
of the number of equivalence classes of convex lattice polygons having
a prescribed area $A$ (we say that two polygons having their vertices
on $\Z^2$ are equivalent if one is the image of the other by an affine
automorphism of $\Z^2$). Arnold obtained bounds of order $A^{1/3}$ up to a $\log A$ factor. 
In 1992, B\'ar\'any and Pach removed this extra factor in \cite{barany_pach_1992}. A extension in higher dimension involving the volume to the power $\frac{d-1}{d+1}$ was achieved by B\'ar\'any and Vershik in \cite{barany_vershik_1992}.
Later, Vershik changed the constraint in this problem and raised the
question of the number, and typical shape, of convex lattice polygons
included in a large box $[-n,n]^2$.
The cornerstone in this problem is the estimation of the number $p(n)$
of convex polygonal chains with vertices in $(\ZZ\cap [0,n])^2$ and
joining $(0,0)$ to $(n,n)$ (or more generally $(n,cn)$ for some $c>0$).
In 1994, a solution to this problem was found independently by
B\'ar\'any~\cite{barany_limit_1995}, Vershik~\cite{MR1275724} and Sinai~\cite{MR1283251}.
Namely, they showed by different methods that, as $n\to\infty$, 
$$
  p(n)=\exp\left[3\kappa^{1/3}n^{2/3}(1+o(1))\right],
  \qquad \text{where }
  \kappa=\frac{\zeta(3)}{\zeta(2)},
$$
and that the limit shape of a typical convex polygonal chain  is the
arc of parabola tangent to the sides of the square,
which maximizes the affine perimeter.

Note that the approach of Sinai was recently made rigorous and extended
by Bogachev and Zarbaliev~\cite{bogachev_zarbaliev_universality_2011}.
This is also this approach that we choose in this paper and we make use
of a proposition of this latter paper.
Here, we go further and fully exploit Sinai's probabilistic model by
giving an exact integral representation of the partition function and
by making a precise asymptotic analysis of it.
After our first post of this paper on arXiv, the cardinality of this set $p(n)$ was recorded by Koutschan in the OEIS as sequence A267862~\cite{oeis}.

The appearance of the values of the Riemann zeta function in the above
formula suggests the arithmetic aspects of the problem.
In this paper, we establish a connection between the combinatorial
analysis of the number of convex chains and the zeros of Riemann's
zeta function.

\begin{theorem}
\label{thm:formula}
Let $p(n)$ denote the number of lattice convex chains with endpoint $(n,n)$. 
As $n$ goes to infinity, 
\begin{equation}\label{eq:equivalent}
    p(n)
    \sim
    \frac{e^{-2\zeta'(-1)}}{(2\pi)^{7/6} \sqrt{3} \kappa^{1/18} n^{17/18}}
    \exp\left[
      3\kappa^{1/3} n^{2/3}
      + I_{\mathrm {crit}}\left(\left(\frac\kappa n\right)^{1/3}\right)
    \right],
\end{equation}
where $\kappa = \zeta(3)/\zeta(2)$ and where the function
$I_{\mathrm {crit}}$ will be defined later by equation \eqref{eq:H_int}.

Moreover, under the assumption that the zeros $\rho$ of the Riemann
zeta function inside the critical strip $0 < \Re(\rho) < 1$ are simple, this function can
be expressed as
\begin{equation*}
  I_{\mathrm {crit}}(\beta)
  =
  \sum_{\rho} \frac{\Gamma(\rho)\zeta(\rho+1)\zeta(\rho-1)}{\zeta'(\rho)\beta^\rho},
\end{equation*}
where the precise meaning of the series will be given  by \eqref{eq:H_sum}.
\end{theorem}

A straightforward corollary of this theorem is that,
if Riemann's Hypothesis was to hold, the oscillatory term
$I_{\mathrm{crit}}((\kappa/n)^{1/3})$ inside the exponential would be
roughly of order $n^{1/6}$.
The two statements are actually equivalent, as we will show in Subsection~\ref{ssec:second-theorem}:

\begin{theorem}
\label{thm:rh}
The condition
\begin{equation}\label{eq:H}\tag{H}
\forall \epsilon>0, \quad \log p(n) =  3\kappa^{\frac13} n^{\frac23}+O(n^{\frac16+\epsilon})
\end{equation}
holds if and only if Riemann's Hypothesis does.
\end{theorem}

In \cite{bodini_asymptotic_2013}, Bodini, Duchon, Jacquot and Mutafchiev
presented a precise asymptotic analysis of digitally convex polyominoes,
based on generating functions and on the saddle point
method. This problem turns out to be strongly related to the enumeration
of convex chains. In the last section, we show how a slight modification
of Sinai's model makes it possible to enumerate polyominoes.

\section{A statistical mechanical model}
\label{sec:correspondence}

We start this paper by reminding the correspondence between lattice
convex chains and non negative integer-valued  functions on the set of
pairs of coprime positive integers.

Let $\Prim$ be the set of primitive vectors, that is to say, the set of
all vectors $(x,y)$ whose coordinates are coprime positive integers,
including the pairs $(0,1)$ and $(1,0)$.
As already used in \Jarnik{}~\cite{MR1544776}, the space of lattice
increasing convex chains starting from the origin is in one-to-one
correspondence with the space \(\Omega\) of nonnegative integer-valued
functions $\omega : \Prim \to \Z_+$ with finite support (that is to say,
\(\omega(\bv)\neq 0\) only for finitely many \(\bv \in \Prim\)):
\begin{itemize}
  \item The function $\omega$ associated to the convex chain
    $(x_i,y_i)_{0 \leq i \leq k}$ is defined for all $\bv \in \Prim$
    by $\omega(\bv) = \gcd(x_{i+1}-x_i,y_{i+1}-y_i)$ if there exists
    $i \in \{0,\dots,k\}$ such that $(x_{i+1}-x_i,y_{i+1}-y_i)$ is
    proportional to $\bv$, and $\omega(\bv) = 0$ otherwise.

  \item The inverse map is obtained by adding up the vectors
  $\omega(\bv)\,\bv$ by increasing slope order.
  In particular, the endpoint of the chain is equal to
  $$
    \sum_{\bv \in \Prim} \omega(\bv)\,\bv.
  $$
\end{itemize}

\subsection{Description of Sinai's model and overall strategy}

We endow the space $\Omega$ with Boltzmann-like probability
measures $\PP_\bbeta$ depending on two parameters $\bbeta = (\beta_1,\beta_2) \in
(0,+\infty)^2$ and characterized by the condition that the random variables
$(\omega(\bv))_{\bv \in \Prim}$ are independent and geometrically
distributed with parameter $e^{- \bbeta \cdot \bv}$, respectively.
In this setting,
\[
    \sum_{\bv \in \Prim} \PP_\bbeta[\omega(\bv) \neq 0] < \infty,
\]
hence, by the Borel--Cantelli lemma,  the function $\omega$ has almost surely finite support and we can write
\[
    \PP_\bbeta(\omega)
    = \prod_{\bv \in \Prim}\left(1-e^{-\bbeta \cdot \bv}\right)e^{-\omega(\bv) \bbeta \cdot \bv}
    = \frac{1}{Z(\bbeta)} \exp\left(- \bbeta \cdot \sum_{\bv \in \Prim} \omega(\bv)\bv\right),
\]
where the \emph{partition function} $Z(\bbeta)$ is given by
\begin{equation}
	\label{eq:partition_function}
    Z(\bbeta) = \prod_{\bv \in \Prim} \left(1 - e^{-\bbeta \cdot \bv}\right)^{-1}
\end{equation}
(the ``$\cdot$" notation denotes the canonical inner product of $\RR^2$).

Since, as noticed above, $\mathbf X(\omega) = \sum_{\bv \in \Prim}
\omega(\bv)\bv$ is the endpoint of the chain corresponding to $\omega$, and $\PP_\bbeta(\omega)$ is 
proportional to $\exp(-\beta \cdot\mathbf X(\omega))$,
the conditional distribution induced by $\PP_\bbeta$
on the set of convex chains ending at $\bn = (n_1,n_2)$ is the uniform
distribution on this latter set. 
By the same argument, we obtain the following formula which will be instrumental
in the proof:
\begin{equation}\label{eq:link}
    \PP_\bbeta\left[\mathbf X = \bn\right] = p(\bn) \frac{e^{-\bbeta\cdot \bn}}{Z(\bbeta)}.
\end{equation}

In order to get a formula for $ p(\bn)$, our strategy is to choose the
two parameters so that $\EE_\bbeta[\mathbf X] = \bn$.
This will indeed lead to an asymptotic equality of $\PP_\bbeta[\mathbf
X = \bn]$ due to a local limit result.
Together with the analysis of the partition function, this local limit
result will constitute the key of the proof.

\subsection{Integral representation of the partition function}

It turns out that the Mellin inversion formula leads to an exact
integral representation of the logarithmic partition function $\log Z$
of Sinai's model in terms of the Euler $\Gamma$ function, the Riemann
$\zeta$ function, and the modified Barnes zeta function
$$
  \chi(s;\bbeta):=
  \sum_{\bv \in \Z_+^2\setminus\{0\}}\frac{1}{(\bbeta \cdot \bv)^s},
  \qquad \Re(s) > 2.
$$
\begin{lemma}\label{lem:logZ}
    For all $\bbeta = (\beta_1,\beta_2) \in (0,+\infty)^2$,
    \begin{equation}
        \label{eq:logZ}
        \log Z(\bbeta)
        =
        \frac{1}{2i\pi} \int_{3-i\infty}^{3+i\infty} \frac{\Gamma(s) \zeta(s+1)}{\zeta(s)}\chi(s;\bbeta)\,ds.
    \end{equation}
   
\end{lemma}

\begin{proof}
    We first take logarithms in \eqref{eq:partition_function} and then expand in Taylor series, 
    \[
        \log Z(\bbeta)
        = -\sum_{\bv \in \Prim} \log\left(1-e^{-\bbeta \cdot \bv}\right)
        = \sum_{\bv \in \Prim}\sum_{m \geq 1} \frac{1}{m} e^{-m \bbeta \cdot \bv}.
    \]

    Now, we make use of the Mellin inversion formula
    \[
        e^{-z} = \frac{1}{2i\pi}\int_{c-i\infty}^{c+i\infty} \Gamma(s) z^{-s} ds,
    \]
    which holds for every real numbers $z>0$ and $c > 0$.
    For all $c>0$, we know, from the exponential decrease of $\Gamma$,
    that its integral on $c+i\RR$ is absolutely convergent.
    Taking $c > 2$ in order to ensure that $\sup_{s \in c +i\RR}
    \sum_{m\geq1} \left| \frac{1}{m^{s+1}}\right| < \infty$ and
      $\sup_{s \in c+i\RR}
      \sum_{\bv\in\Prim}\left|\frac{1}{(\bbeta\cdot\bv)^{s}}\right| <
      \infty$, we can apply Fubini's theorem and write
    \begin{align*}
        \log Z(\bbeta) & = \frac{1}{2i\pi}\int_{c-i\infty}^{c+i\infty} \sum_{\bv \in \Prim} \sum_{m \geq 1} \frac{1}{m}\frac{\Gamma(s)}{(m\bbeta \cdot \bv)^s}ds\\
       & = \frac{1}{2i\pi}\int_{c-i\infty}^{c+i\infty} \zeta(s+1)\Gamma(s)\sum_{\bv \in \Prim} \frac{1}{(\bbeta \cdot \bv)^s}ds.
   \end{align*}

   To conclude the proof of \eqref{eq:logZ}, notice that the partition $\Z_+^2\setminus\{0\}
   = \bigsqcup_{d\ge 1} (d\Prim)$ translates into the identity of
   Dirichlet series
   \[
      \sum_{\bv \in \Z_+^2\setminus\{0\}} \frac{1}{(\bbeta \cdot \bv)^s} = \sum_{d \geq 1}\frac{1}{d^s}\sum_{\bv \in \Prim} \frac{1}{(\bbeta \cdot \bv)^s}.\qedhere
   \]
\end{proof}

Our main contribution is to exploit this integral representation in
order to push further the estimates of $\log Z$, leading to a precise
estimate of $p(n,n)$.  As Sinai does in his paper, we could also get
precise asymptotics for $p(n,cn)$, but for the sake of readability of
the paper, we limit our scope to $p(n,n)$, which leads to a choice of
equal parameters $\beta_1=\beta_2=\beta$.
In this case, we observe by elementary manipulations, that for $\Re(s) > 2$,
\begin{equation}\label{eq:chi_manip}
  \chi(s;(\beta,\beta))
   = \sum_{\bv \in \Z_+^2\setminus\{0\}} \frac{1}{(\beta v_1 + \beta v_2)^s}
   = \sum_{n \geq 1} \frac{n+1}{\beta^s n^s}\\
   = \frac{\zeta(s-1)+\zeta(s)}{\beta^s}.
\end{equation}
Therefore, equation \eqref{eq:logZ} becomes
\begin{equation}\label{eq:logZbb}
  \log Z(\beta,\beta)
  =
  \frac{1}{2i\pi} \int_{3-i\infty}^{3+i\infty} \frac{\Gamma(s) \zeta(s+1)(\zeta(s-1)+\zeta(s))}{\zeta(s)\beta^s}\,ds.
\end{equation}

\section{Analysis of the partition function}
We start by giving the first order estimates of $\log Z(\bbeta)$ and of its partial derivatives. For this purpose, we make use of the meromorphic continuation of the modified Barnes zeta function $\chi$, which is detailed in the Appendix.

\begin{lemma}\label{lem:cumulants}
For all nonnegative integers $k_1, k_2$, for all $\epsilon>0$, and all
$\bbeta=(\beta_1, \beta_2)\in (0,+\infty)^2$, such that
$\epsilon < \frac{\beta_1}{\beta_2}<\frac1{\epsilon}$,
\[
  \frac{\partial^{k_1+k_2}}{\partial\beta_1^{k_1}\partial\beta_2^{k_2}}\log Z(\beta_1, \beta_2)
  \underset{\bbeta\to0}\sim
  (-1)^{k_1+k_2}\frac{\zeta(3)}{\zeta(2)}\frac{k_1!k_2!}{\beta_1^{k_1+1}\beta_2^{k_2+1}}.
\]
\end{lemma}

\begin{proof}

We apply formula \eqref{eq:logZ} and shift the line of integration to the left by using the residue theorem.
Note that, by Corollary~\ref{Barnes} of the Appendix, all
the integrated functions can be meromorphically continued, that the
only pole in the region $\Re(s)\geq1$ lies at $s=2$, and that for all
$\delta\in(0,1)$,
\[
  \log Z(\bbeta)
  =
  \frac{\zeta(3)}{\zeta(2)}\frac1{\beta_1\beta_2}+\frac{1}{2i\pi} \int_{1+\delta-i\infty}^{1+\delta+i\infty} \frac{\Gamma(s) \zeta(s+1)}{\zeta(s)}\chi(s;\bbeta)\,ds.
\]

We can take the iterated derivatives formally in the previous
equality, since we control the derivatives of $\chi(s;\bbeta)$ by
Corollary~\ref{Barnes}, and since both functions $\zeta(s+1)$ and
$1/\zeta(s)$ are bounded on the line $1+\delta+i\RR$.
\end{proof}

In the special case $\beta_1=\beta_2=\beta$, we derive a much more
precise asymptotic series expansion of $\log Z(\beta, \beta)$ and of
its derivatives in $\beta$.
The proof is based on formula~\eqref{eq:logZbb} and on the residue
theorem again.

\begin{lemma}
\label{lem:expansion}
Let $A$ be a positive number such that the contour $\gamma$ defined
below surrounds all the zeros of $\zeta$ inside the critical strip.

The contour $\gamma$ is defined as the union of the following oriented
paths: on the right side the curve $\gamma_{\mathrm{right}}(t) = 1 -
\frac{A}{\log (2 + |t|)} + it$ for $t$ going from $-\infty$ to $+\infty$
and on the left side the curve  $\gamma_{\mathrm{left}}(t) = \frac{A}{\log
(2 + |t|)} + it$ for $t$ going from $+\infty$ to $-\infty$
(see Figure~\ref{fig:contours}).

We define the functions $I_{\mathrm{crit}}$ and $I_{\mathrm{err}}$
respectively by
\begin{equation}\label{eq:H_int}
    I_{\mathrm{crit}}(\beta) = \frac{1}{2i\pi} \int_\gamma \frac{\Gamma(s)\zeta(s+1)(\zeta(s-1)+\zeta(s))}{\zeta(s)\beta^s}ds
\end{equation}
and
\[
    I_{\mathrm{err}}(\beta) = \frac{1}{2i\pi} \int_{-\frac{1}{2} -i\infty}^{-\frac{1}{2}+i\infty} \frac{\Gamma(s)\zeta(s+1)(\zeta(s-1)+\zeta(s))}{\zeta(s)\beta^s}ds.
\]

\begin{enumerate}[(i)]
  \item For all $k\geq0$, when $\beta$ goes to 0, the $k$-th derivative
  of $I_{\mathrm{crit}}(\beta)$ is $o(\beta^{-k-1})$.

  \item When $\beta$ goes to 0, the function $I_{\mathrm{err}}(\beta)$
is of order $O(\beta^{1/2})$ and $I_{\mathrm{err}}'(\beta) = o(1/\beta)$.
 
 \item For all $\beta>0$, 
  \begin{equation}
      \label{eq:sf_log}
      \log Z(\beta, \beta) = \frac{\zeta(3)}{\zeta(2)}\frac{1}{\beta^2} +
      I_{\mathrm{crit}}(\beta) +
      \frac{7}{6} \log \frac{1}{\beta} + C  + I_{\mathrm{err}}(\beta)
  \end{equation}
  with $C = -2 \zeta'(-1) - \frac{1}{6} \log(2\pi)$.
\end{enumerate}
\end{lemma}

\definecolor{qqqqff}{rgb}{0,0,1}

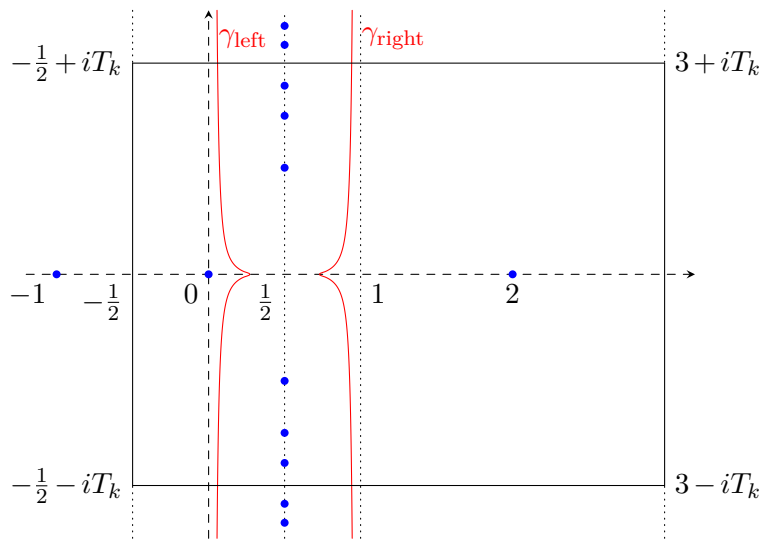
\begin{figure}
  \begin{center}
    \begin{tikzpicture}[line cap=round,line join=round,x=2.0cm,y=0.1cm,>=stealth]
    \draw[->,dashed] (-1.2,0) -- (3.2,0);
    \draw[->,dashed] (0,-35) -- (0,35);
    \draw [dotted] (1,-35) -- (1,35);
    \draw [dotted] (0.5,-35) -- (0.5,35);
    \draw [dotted] (-0.5,35)-- (-0.5,28);
    \draw [dotted] (-0.5,-35)-- (-0.5,-28);
    \draw (-0.5,-28) -- (-0.5,28);
    \draw (3,-28)-- (3,28);
    \draw [dotted] (3,35)-- (3,28);
    \draw [dotted] (3,-35)-- (3,-28);
    \draw[color=red] plot [smooth] file {gamma_right.table};
    \draw[color=red] plot [smooth] file {gamma_left.table};
    \draw (3,28) node[anchor=west] {$3+iT_k$};
    \draw (3,-28) node[anchor=west] {$3-iT_k$};
    \draw (-0.5,28) node[anchor=east] {$-\frac{1}{2}+iT_k$};
    \draw (-0.5,-28) node[anchor=east] {$-\frac{1}{2}-iT_k$};
    \draw[color=red] (0.94,34.04) node[anchor=north west] {$\gamma_{\mathrm{right}}$};
    \draw (-1,0) node[anchor=north east] {$-1$};
    \draw (-0.5,0) node[anchor=north east] {$-\frac{1}{2}$};
    \draw (2,0) node[anchor=north] {$2$};
    \draw (0,0) node[anchor=north east] {$0$};
    \draw (0.5,0) node[anchor=north east] {$\frac{1}{2}$};
    \draw (1,0) node[anchor=north west] {$1$};
    \draw[color=red] (0,33.79) node[anchor=north west] {$\gamma_{\mathrm{left}}$};
    \draw (3,-28)-- (-0.5,-28);
    \draw (3,28)-- (-0.5,28);
    \begin{scriptsize}
    \fill [color=qqqqff] (2,0) circle (1.5pt);
    \fill [color=qqqqff] (0,0) circle (1.5pt);
    \fill [color=qqqqff] (-1,0) circle (1.5pt);
    \fill [color=qqqqff] (0.5,14.13) circle (1.5pt);
    \fill [color=qqqqff] (0.5,21.02) circle (1.5pt);
    \fill [color=qqqqff] (0.5,25.01) circle (1.5pt);
    \fill [color=qqqqff] (0.5,30.42) circle (1.5pt);
    \fill [color=qqqqff] (0.5,-14.13) circle (1.5pt);
    \fill [color=qqqqff] (0.5,-21.02) circle (1.5pt);
    \fill [color=qqqqff] (0.5,-25.01) circle (1.5pt);
    \fill [color=qqqqff] (0.5,-30.42) circle (1.5pt);
    \fill [color=qqqqff] (0.5,32.94) circle (1.5pt);
    \fill [color=qqqqff] (0.5,-32.94) circle (1.5pt);
    \end{scriptsize}
    \end{tikzpicture}
  \end{center}
  \caption{Localization of poles (in blue) and contours of integration with which the residue theorem is applied to get the asymptotic expansion of $\log Z(\beta,\beta)$ as $\beta \to 0$.}
  \label{fig:contours}
\end{figure}

\begin{proof}
The existence of a suitable number $A$ is a consequence of
\cite{titchmarsh_riemann_1986}*{Theorem 3.8}, and the fact that the set
of nontrivial zeros of $\zeta$ is symmetric with respect to the vertical
line $\Re(s)=\frac12$.
In addition, the integrals defining $I_{\mathrm{crit}}(\beta)$ and
$I_{\mathrm{err}}(\beta)$ are convergent.
Indeed, along the path $\gamma$, the function $1/\zeta(s)$ is $O(\log(|\Im
(s)|))$ by formula (3.11.8) in \cite{titchmarsh_riemann_1986}.
Moreover, this domination makes it possible to apply Lebesgue's dominated
convergence theorem to show that the function $I_{\mathrm{crit}}(\beta)$
is $o(\beta^{-1})$ as $\beta \to 0$.
The same domination allows us to differentiate $I_{\mathrm{crit}}$
under the integral sign as many times as needed and yields that the $k$-th derivative of
$I_{\mathrm{crit}}(\beta)$ is $o(\beta^{-k-1})$.

Assertion \textit{(ii)} is easier since the function $1/\zeta(s)$ is
bounded on the vertical line $\Re(s)=-\frac12$.

We turn now to the proof of \textit{(iii)}.
As in the previous lemma, the strategy is to start from the integral
representation  \eqref{eq:logZbb} of $\log Z$. Here, in order to get a
sharper asymptotic expansion, we introduce a contour whose left side is
the line $\Re(s)=-\frac12$.
Since such a contour crosses the critical strip where the denominator
$\zeta(s)$ has zeros, we will use the following result of Valiron
\cite{titchmarsh_riemann_1986}*{Theorem 9.7}: there exists $\alpha > 0$
and a sequence $(T_k)$ such that for all $k \in \N$, $k < T_k < k+1$
and $|\zeta(s)| > |\Im(s)|^{-\alpha}$ uniformly for all $s$ such that
$|\Im(s)| = T_k$ and $-1 \leq \Re(s) \leq 2$.
Therefore, if one applies the residue theorem with the positevely oriented rectangle of vertices
$3\mp iT_k$ and $ -\frac{1}{2}\pm iT_k$, and lets $k$ tends to $+\infty$,
the contributions of the horizontal segments tend to $0$ and one gets
that $\log Z(\beta,\beta)$ is the sum of $I_{\mathrm{crit}}(\beta)$,
$I_{\mathrm{err}}(\beta)$ and of the residues of the integrated function
between the lines $\Re(s)=-\frac12$ and $\Re(s)=3$ which are outside of
the contour $\gamma$. Note that, by definition of the contour $\gamma$, there is no zero of $\zeta$ in this region.
See Figure~\ref{fig:contours} for a landscape of the proof.

Since the singularity at $s = 1$ is cancelled by the presence of
$\zeta(s)$ in the denominator, these residues come only from the simple
pole at $s=2$ and from the double pole at $s=0$.
It is straightforward to observe that the residue at $s=2$ equals
$(\zeta(3)/\zeta(2))\beta^{-2}$.
A more involved but yet elementary computation, based on the Laurent series expansions of $\Gamma$ and $\zeta$ (see \cite{titchmarsh_riemann_1986}), yields
\[
    \mathrm{Res}_{s=0} \left(\frac{\Gamma(s) \zeta(s+1)(\zeta(s-1)+\zeta(s))}{\zeta(s)\beta^s}\right) = \frac{7}{6}\log \frac{1}{\beta} - 2\zeta'(-1)-\frac{1}{6}\log(2\pi).
\]
The announced formula \eqref{eq:sf_log} is thus proven.
\end{proof}

The residue theorem applied to the above growing contours also leads to
the following alternative expression:
\begin{equation}
    \label{eq:H_sum}
    I_{\mathrm{crit}}(\beta) = 
    \lim_{k \to \infty} \sum_{|\Im (\rho)| < T_k} \frac{\Gamma(\rho)\zeta(\rho+1)\zeta(\rho-1)}{\zeta'(\rho) \beta^\rho},
\end{equation}
where $\rho$ runs through the zeros of $\zeta$ with $0 < \Re(\rho) <
1$. Here we have assumed for notational simplicity that these zeros
have multiplicity $1$ but analogous formulas are available for arbitrary multiplicities.

Remark that in the previous arguments, one could push the left-side of
the rectangular contour of integration as far as needed to the left in
order to obtain complete asymptotic series expansions.

\section{Proof of the theorems}

\subsection{First theorem}

Let us recall that the endpoint is defined by
\[
  \mathbf X(\omega)
  = (X_1(\omega),X_2(\omega))
  = \sum_{\bv \in \Prim} \omega(\bv)\,\bv.
\]
We look for a parameter $(\beta_1, \beta_2)$ such that
$\EE_{\bbeta}(X_1)=\EE_{\bbeta}(X_2)=n$.
If one chooses $\beta_1=\beta_2=\beta$, then
$\EE_{\bbeta}(X_1)=\EE_{\bbeta}(X_2)=\frac 1 2 \EE_{\bbeta}(X_1+X_2)$.
It now remains to find $\beta > 0$ satisfying $\EE_{\bbeta}(X_1+X_2) = 2n$.
A direct computation shows that $\EE_{\bbeta}(X_1+X_2)=-\frac{d}{d\beta}
\log Z(\beta, \beta)$.
By monotonicity of the function $\beta \mapsto \frac{d}{d\beta}\log
Z(\beta,\beta)$, we can find $\beta > 0$ depending on $n$ such that
$$
  \frac{d}{d\beta}\log Z(\beta,\beta) = -2n.
$$
>From now on, $\beta_1$ and $\beta_2$ are going to be chosen equal to
the unique solution $\beta$ of this equation, depending on $n$.
Note that $\beta$ tends to $0$ as $n$ tends to $\infty$.

As a consequence of statements \textit{(ii)} and \textit{(iii)} of
Lemma~\ref{lem:expansion}, the following asymptotic expansion holds:
\[
2n = -\frac{d}{d\beta}\log Z(\beta, \beta) = \frac{2\kappa}{\beta^3} -
    I_{\mathrm{crit}}'(\beta) + \frac{7}{6\beta} + o\left(\frac1\beta\right),
\]
where $\kappa = \zeta(3)/\zeta(2)$.
In order to estimate the error made when replacing the argument $\beta$ of
$I_{\mathrm{crit}}'$ by its first order approximation $(\kappa/n)^{1/3}$,
we use the estimates of Lemma~\ref{lem:expansion}, statement \textit{(i)},
for $I_{\mathrm{crit}}'$ and $I_{\mathrm{crit}}''$.
We obtain therefore
\begin{equation}
    \label{eq:sf_beta}
    \frac{1}{\beta^3} = \frac{n}{\kappa} + \frac{1}{2\kappa} I_{\mathrm{crit}}'\left(\left(\frac{\kappa}{n}\right)^{1/3}\right) - \frac{7}{12 \kappa} \left(\frac{n}{\kappa}\right)^{1/3} + o(n^{1/3}).
\end{equation}

The covariance matrix of the random vector $\mathbf X = (X_1,X_2)$
is equal to the Hessian matrix of $\log Z(\beta_1,\beta_1)$ at
$(\beta,\beta)$, as seen by a straightforward computation.
For the calibrated parameter $\beta$, Lemma~\ref{lem:cumulants} implies
that its determinant is asymptotically equal to
 \[
  \begin{vmatrix}
    \frac{2\kappa}{\beta^4} & \frac{\kappa}{\beta^4}\\
    \frac{\kappa}{\beta^4} & \frac{2\kappa}{\beta^4}
  \end{vmatrix}
  = \frac{3\kappa^2}{\beta^8}
  \sim \frac{3 n^{8/3}}{\kappa^{2/3}}.
\]
Therefore, a local limit theorem, which is a simplified version of Theorem
1.3 in \cite{bogachev_zarbaliev_universality_2011} (in the case $r=1$,
with the notation of the paper), gives
\begin{equation}\label{eq:sf_llt}
    \PP_\bbeta[X_1= n, X_2 = n] \sim \frac{\kappa^{1/3}}{2\pi \sqrt{3} n^{4/3}}.
\end{equation}
>From formula~\eqref{eq:link}, $\log p(n) = 2n\beta + \log Z(\beta,\beta) + \log \PP_\bbeta[X_1 =n, X_2 = n]$.
Hence, gathering \eqref{eq:sf_log}, \eqref{eq:sf_beta}, \eqref{eq:sf_llt}, we are able to state

\begin{align*}
   \log p(n) & = 2n\beta + \log Z(\beta,\beta) + \log \frac{\kappa^{1/3}}{2\pi\sqrt 3} - \frac 4 3 \log n + o(1)\\
   & =  3\kappa^{1/3} n^{2/3}
     + I_{\mathrm{crit}}\left(\left(\frac \kappa n\right)^{1/3}\right)
     - \frac{17}{18}\log n
     + \log \left( \frac{e^{-2\zeta'(-1)}}{(2\pi)^{7/6} \sqrt{3} \kappa^{1/18}}\right)
     + o(1).
\end{align*}
This concludes the proof of Theorem~\ref{thm:formula}.

\subsection{Second theorem}\label{ssec:second-theorem}

As we observed in the proof of statement \textit{(i)} of
Lemma~\ref{lem:expansion}, the fact that the additional term
$I_{\mathrm{crit}}((\kappa/n)^{1/3})$ is at most of order $o(n^{1/3})$
follows from the existence of large zero-free regions of the Riemann
zeta function in the critical strip.
This estimation can be considerably improved if one assumes Riemann's
Hypothesis, namely that all zeros of the zeta function in the critical
strip lie actually on the line $\Re(s) = \frac{1}{2}$.
Under this assumption, the oscillating term
$I_{\mathrm{crit}}((\kappa/n)^{1/3})$ is at most of order
$O(n^{\frac16+\epsilon})$ for all $\epsilon > 0$.
Indeed, the right-side $\gamma_{\mathrm{right}}$ of the contour of integration in formula \eqref{eq:H_int} can be replaced by a vertical line arbitrarily close to the critical line. This writing implies that, for all $\epsilon>0$, $I_{\mathrm{crit}}(\beta)=O(\beta^{-\frac12-\epsilon})$ as $\beta$ goes to 0. 
Hence, Riemann's Hypothesis implies \eqref{eq:H}.

Conversely, let us assume that hypothesis~\eqref{eq:H} holds. By Theorem \ref{thm:formula}, for all $\epsilon>0$, $I_{\mathrm{crit}}((\kappa/n)^{\frac13}))=O(n^{\frac16+\epsilon})$.
Since $I_{\mathrm{crit}}'(\beta)=o(\beta^{-2})$ by statement \textit{(i)} of Lemma~\ref{lem:expansion}, this implies that for all $\epsilon>0$, $I_{\mathrm{crit}}(\beta)=O(\beta^{-\frac12-\epsilon})$.
Statement \textit{(iii)} of Lemma~\ref{lem:expansion} then yields 
\begin{equation}\label{eq:H'}\tag{H'} 
\forall \epsilon>0,\quad \log Z(\beta, \beta)-\frac\kappa{\beta^2}=O\left(\frac1{\beta^{\frac12+\epsilon}}\right)
\end{equation}
as $\beta\to0$. 
Now, we compute the Mellin transform of $\log Z(\beta,\beta)=\sum_{\bv \in \Prim}\sum_{m \geq 1} \frac{1}{m} e^{-m \beta(v_1+v_2)}$ for $\Re(s)>2$ by manipulations similar to the proof of Lemma~\ref{lem:logZ}:
$$\int_0^{+\infty} \log Z(\beta,\beta)\beta^{s-1}d\beta= \frac{\Gamma(s) \zeta(s+1)(\zeta(s-1)+\zeta(s))}{\zeta(s)}.$$
We obtain therefore the following identity for $\Re(s)>2$:
\begin{multline*}
	\frac{\Gamma(s) \zeta(s+1)(\zeta(s-1)+\zeta(s))}{\zeta(s)}-\frac\kappa{s-2}=\\
	\int_0^1 \left(\log Z(\beta,\beta)-\frac\kappa{\beta^2}\right)\beta^{s-1}d\beta +\int_1^{+\infty} \log Z(\beta,\beta)\beta^{s-1}d\beta.
\end{multline*}
Condition \eqref{eq:H'} implies that the right-hand side function is holomorphic in the region $\Re(s)>1/2$, hence defining a holomorphic continuation of the left-hand side function in the same region.
Since $\Gamma(s)$, $\zeta(s+1)$ and $\zeta(s-1)$ do not vanish in the region $0 < \Re(s) < 1$, this prevents $\zeta(s)$ from having zeros in the region $1/2 < \Re(s) < 1$. As a consequence of the functional equation satisfied by $\zeta$ (see \cite{titchmarsh_riemann_1986}),  the distribution of the nontrivial zeros is symmetric with respect to the critical line. Hence, Riemann's Hypothesis is implied by \eqref{eq:H}.

\section{Comments}
    
\subsection{Numerical considerations}

In view of numerical computations, one can truncate the series defining
$I_{\mathrm{crit}}$ in $\eqref{eq:H_sum}$ at the two conjugate zeros of
the zeta function with smallest imaginary part, which is approximately
$\frac12 \pm i14.1347$.
Indeed, the next zeros have an imaginary part around $21.0220$ and the
exponential decrease of the $\Gamma$ function makes this next term around
$10^4$ times smaller than the first one.
Therefore, with a relative precision of $10^{-4}$,
\[
  I_{\mathrm{crit}}(\beta) \approx
  \frac{6.0240\cdot10^{-11} \cos(14.1347 \log \beta) + 9.5848\cdot10^{-10}\sin(14.1347 \log \beta)}{\sqrt{\beta}}
\]
for usual values of $\beta = (\kappa/n)^{1/3}$.
In addition, one can see from this approximation that for a large range of
values of $n$, the importance of the oscillatory term in the exponential
part of \eqref{eq:equivalent} is actually extremely small in comparison
with the polynomial pre-factor. 
It is therefore likely that formula $$\frac{e^{-2\zeta'(-1)}}{(2\pi)^{7/6} \sqrt{3} \kappa^{1/18} n^{17/18}}
    \exp\left[
      3\kappa^{1/3} n^{2/3}\right]$$ gives a sharp estimation of $p(n)$ as long as $n$ is ``large'' but less than $10^{60}$. A vague hint of this is given by the exact value of $p(100)=26878385993387721255010$ computed by Vaclav Kotesovec \cite{oeis} which differs from the analytic formula leading to
$ 2.4\cdot 10^{22}$ only by $10\%$. 
\subsection{Digitally convex polyominoes}

We present here a modification of our model which leads to
an asymptotic analysis of digitally convex polyominoes as in
\cite{bodini_asymptotic_2013}.
Our probabilistic approach notably differs from the saddle point method used in
this paper. We were not able to follow all the steps of the computations
in \cite{bodini_asymptotic_2013}, and we obtain eventually a slightly
different result.

Let us first recall that a digitally convex polyomino is the set of all
cells of $\ZZ^2$ included in a bounded convex region of the plane.
The contour of a digitally convex polyomino can be decomposed into four
specifiable sub-paths through the standard decomposition of polyominoes.
In \cite{bodini_asymptotic_2013}, the authors focus on one of these
paths, namely the one joining the rightmost lowest point to the highest
rightmost one. More precisely, they want to find the asymptotics of the
number $\widetilde p(n)$ of such paths with total length $n$.

Taking the convex hull of such a path defines a one-to-one correspondence
with (increasing) convex chains having no horizontal segment.
For this reason, we slightly modify the model by changing the set $\Prim$
into $\widetilde \Prim = \Prim\setminus \{(1,0)\}$ and the probability
measure on configurations which are now integer-valued functions on
$\widetilde \Prim$:
\[
  \widetilde\PP_\bbeta(\omega)
   = \prod_{\bv \in \widetilde\Prim}\left(1-e^{-\bbeta \cdot \bv}\right)e^{-\omega(\bv) \bbeta \cdot \bv}
   = \frac{1}{\widetilde Z(\bbeta)} \exp\left(- \bbeta \cdot \sum_{\bv \in \widetilde\Prim} \omega(\bv)\bv\right)
\]
where $\bbeta = (\beta,\beta)$ and
\[
    \widetilde Z(\bbeta) = \prod_{\bv \in \widetilde\Prim} \left(1 - e^{-\bbeta \cdot \bv}\right)^{-1}.
\]
Again, the ending point is equal to
\[
    \mathbf X(\omega) = (X_1(\omega),X_2(\omega)) = \sum_{\bv \in \widetilde\Prim} \omega(\bv) \bv.
\]
Moreover the connection between combinatorics and the probabilistic
model is encompassed in the formula
\[
    \widetilde \PP_\bbeta[X_1 + X_2 = n] = \frac{\widetilde p(n)}{\widetilde Z(\beta,\beta)}e^{- \beta n}.
\]
A computation similar to \eqref{eq:chi_manip} shows that the integral
representation \eqref{eq:logZbb} of the logarithmic partition function
is changed into
\[
    \log \widetilde Z(\beta,\beta) = \frac{1}{2i\pi} \int_{c-i\infty}^{c+i\infty} \frac{\Gamma(s) \zeta(s+1)\zeta(s-1)}{\zeta(s)\beta^s}\,ds,
\]
which leads to the following asymptotic expansion as $\beta$ tends to $0$:
\[
    \log \widetilde Z(\beta,\beta) = \frac{\kappa}{\beta^2} + I_{\mathrm{crit}}(\beta) + \frac 1 6 \log\frac 1 \beta + C + o(1),
\]
where the constants $\kappa$ and $C$ are still $\zeta(3)/\zeta(2)$ and $-2\zeta'(-1) - \frac 1 6 \log(2\pi)$.
Both formulas appeared already in the proof of Proposition 2.2 of \cite{bodini_asymptotic_2013}.
In order to have the correct calibration $\EE_\bbeta[X_1+X_2] = n$ for the
length of the path, we choose $\beta > 0$ depending on $n$ such that
$-\frac{d}{d\beta}\log \widetilde Z(\beta,\beta) = n$.
This leads to the following analogue to \eqref{eq:sf_beta}, as $\beta\to0$:
\[
\frac{1}{\beta^3} = \frac{n}{2\kappa}
    + \frac{1}{2\kappa} I_{\mathrm{crit}}'\left(\left(\frac{2\kappa}{n}\right)^{1/3}\right)
    - \frac{1}{12 \kappa} \left(\frac{n}{2\kappa}\right)^{1/3} + o(n^{1/3}).
\]
The variance $\sigma^2_n$ of $X_1+X_2$  can be computed by applying to the vector $(1,1)$ the quadratic form given by the covariance matrix of $(X_1,X_2)$ whose asymptotics at the first order coincide with the ones computed in last section:
 
 \[
   \sigma_n^2 \sim 
 \begin{pmatrix}
    1 & 1
  \end{pmatrix}
  \begin{pmatrix}
    \frac{2\kappa}{\beta^4} & \frac{\kappa}{\beta^4}\\
    \frac{\kappa}{\beta^4} & \frac{2\kappa}{\beta^4}
  \end{pmatrix}
   \begin{pmatrix}
    1 \\ 1
  \end{pmatrix}
  = \frac{6\kappa}{\beta^4}
  \sim \frac{3 n^{4/3}}{2^{1/3}\kappa^{1/3}}.
\]
As in the previous model, a local limit result applies to the discrete one-dimensional random variable $X_1+X_2$: the probability that this random variable is equal to its mean is asymptotically equal to
\[
    \PP_\bbeta[X_1+X_2 = n] \sim \frac{1}{\sqrt{2\pi \sigma_n^2}}\sim
    \frac{1}{\sqrt{3\pi}}\left(\frac{\kappa}{4}\right)^{\frac16}\frac 1{n^{\frac23}}.
\]
We do not provide here a complete proof of this estimate, but this can be shown by checking the assumptions of \cite{bureaux_2014}*{Proposition 7.1}, which  is close from what was done in \cite{bogachev_zarbaliev_universality_2011}.
We obtain finally, as $n\to\infty$,
\[
    \widetilde p(n) \sim
    \left(\frac{\kappa}{4}\right)^{\frac5{18}}\frac{e^{-2\zeta'(-1)}}{\sqrt 3 \pi^{\frac 2 3}}\frac{1}{n^{\frac{11}{18}}}
    \exp\left[
      3 \left(\frac{\kappa}{4}\right)^{\frac 1 3}n^\frac{2}{3} + I_{\mathrm{crit}}\left(\left(\frac{2 \kappa}{n}\right)^{\frac 1 3}\right)
    \right].
\]
This amends the constant in the prefactor term of \cite{bodini_asymptotic_2013}*{Proposition 2.5}.

\appendix

\section{Analytic continuation of the Barnes function}

We introduce here the Barnes zeta function defined for all $\bbeta=(\beta_1,\beta_2)\in[0,+\infty)^2$, $w>0$ and $s\in\CC$ in the region $\Re(s)>2$ by 
\[
  \zeta_2(s,w;\bbeta):= \sum_{v_1,v_2\geq0}\frac1{(w+\beta_1v_1+\beta_2v_2)^s}.
\]

\begin{lemma}\label{lem:barnes}
The meromorphic continuation of $\zeta_2(s,w ; \bbeta)$ to the half-plane $\Re(s) > 1$ is given by
\begin{align*}
\zeta_2(s,w ;\bbeta) =& \frac{1}{\beta_1\beta_2}\frac{w^{-s+2}}{(s-1)(s-2)} + \frac{(\beta_1 + \beta_2)w^{-s+1}}{2\beta_1\beta_2(s-1)} + \frac{w^{-s}}{4}\\
& - \frac{\beta_2}{\beta_1} \int_0^{+\infty} \frac{\{y\}-\frac{1}{2}}{(w+\beta_2 y)^s}\,dy - \frac{\beta_1}{\beta_2} \int_0^{+\infty} \frac{\{x\}-\frac{1}{2}}{(w+\beta_1 x)^s}\,dx\\
& - s\frac{\beta_2}{2} \int_0^{+\infty} \frac{\{y\}-\frac{1}{2}}{(w+\beta_2 y)^{s+1}}dy
- s\frac{\beta_1}{2} \int_0^{+\infty} \frac{\{x\}-\frac{1}{2}}{(w+\beta_1 x)^{s+1}}dx\\
& + s(s+1)\beta_1\beta_2 \int_0^{+\infty}\int_0^{+\infty} \frac{(\{x\} - \frac{1}{2})(\{y\}-\frac{1}{2})}{(w+\beta_1 x +\beta_2 y)^{s+2}}\,dxdy.
\end{align*}
\end{lemma}

\begin{proof}
We need only prove that the equality holds for $\Re(s) > 2$ since
the right-hand side of the equation is meromorphic in $\Re(s) > 1$
with a single pole at $s = 2$.
Let $\{x\} = x - \lfloor x \rfloor$ denote the fractional part of $x$.
We apply the Euler--Maclaurin formula to the partial summation
defined by $F(x) = \sum_{n_2\geq0} (w +\beta_1 x + \beta_2 n_2)^{-s}$,
leading to
\[
  \sum_{n_1 \geq 1} F(n_1) = \int_0^{\infty} F(x)\,dx - \frac{F(0)}{2} + \int_0^\infty (\{x\}-\frac{1}{2})F'(x)\,dx.
\]
We use again the Euler-Maclaurin formula for each of the summations in $n_2$. 
\end{proof}

We state here the results we need on the modified Barnes zeta function
$\chi$.

 \begin{corollary} \label{Barnes}
For all $\bbeta=(\beta_1,\beta_2)\in[0,+\infty)^2$, the function
$\chi(s;\bbeta)=\sum_{\bv \in \Z_+^2\setminus\{0\}}\frac{1}{(\bbeta
\cdot \bv)^s}$ can be meromorphically continued to the region $\Re (s)>1$.

\begin{enumerate}[(i)]

\item The function $\chi(s;\bbeta)$ has a unique pole at $s=2$ which is
simple and the residue is equal to $\frac1{\beta_1\beta_2}$.

\item For all nonnegative integers $k_1, k_2$, for all $\delta\in
(0,1)$, for all $\epsilon>0$,  there exists a constant $C>0$
such that for all $\bbeta=(\beta_1, \beta_2)\in (0,+\infty)^2$
satisfying $\epsilon < \frac{\beta_1}{\beta_2}<\frac1{\epsilon}$,
for all $s$ such that $1+\delta \leq \Re (s)\leq 3$,
$$
  \left|\frac{\partial^{k_1+k_2}}{\partial\beta_1^{k_1}\partial\beta_2^{k_2}}\chi(s;\bbeta)\right|\leq\frac{C|s|^C}{|\bbeta|^{k_1+k_2+\Re(s)}|s-2|}.
$$
\end{enumerate}

\end{corollary}

\begin{proof}
Elementary manipulations on sums yield 
$$\chi(s;\bbeta)=\frac{\zeta(s)}{\beta_1^s}+\frac{\zeta(s)}{\beta_2^s}+\zeta_2(s,\beta_1+\beta_2;\bbeta).$$
We then apply Lemma~\ref{lem:barnes} and observe that the integral terms
on the right-hand side of the equation can be differentiated formally.
\end{proof}

\subsection*{Acknowledgements} The authors are pleased to thank Axel
Bacher, Philippe Duchon and Jean-Fran{\c c}ois Marckert for bringing to
their knowledge the reference \cite{bodini_asymptotic_2013} on digitally
convex polyominoes. They also want to thank warmly Christoph Koutschan for introducing
the sequence $p(n)$ in the OEIS, and Vaclav Kotesovec for computing its first 100 values.
Special thanks go to the anonymous referee who had a very careful reading of the manuscript.

The authors  finally thank ANR project PRESAGE for funding and organizing stimulating meetings around the subject.

\def\cprime{$'$} \def\cprime{$'$} \def\cprime{$'$}

\begin{dajauthors}
\begin{authorinfo}[J. Bureaux]
Julien Bureaux\\
Universit\'e Paris Nanterre\\
  200 avenue de la R\'epublique 92001 Nanterre, France\\
  \href{mailto:julien.bureaux@math.cnrs.fr}{julien.bureaux@math.cnrs.fr}
\end{authorinfo}
\begin{authorinfo}[N. Enriquez]
Nathana\"el Enriquez\\
Universit\'e Paris Nanterre\\
  200 avenue de la R\'epublique 92001 Nanterre, France\\
  \smallskip
  Universit\'e Pierre--et--Marie--Curie, LPMA\\
 4 place Jussieu 75005 Paris, France\\
  \href{mailto:nathanael.enriquez@u-paris10.fr}{nathanael.enriquez@u-paris10.fr}
\end{authorinfo}
\end{dajauthors}

\end{document}